\documentclass[10pt]{article}
\usepackage{amsmath}
\usepackage{amssymb}
\usepackage{amsthm}
\usepackage{mathrsfs}

\begin{document}
\title{Boundedness of intrinsic square functions on the weighted weak Hardy spaces}
\author{Hua Wang \footnote{E-mail address: wanghua@pku.edu.cn.}\\
\footnotesize{Department of Mathematics, Zhejiang University, Hangzhou 310027, China}}
\date{}
\maketitle

\begin{abstract}
In this paper, by using the atomic decomposition theorem for weighted weak Hardy spaces, we will show the boundedness properties of intrinsic square functions including the Lusin area integral, Littlewood-Paley $g$-function and $g^*_\lambda$-function on these spaces.\\
MSC(2010): 42B25; 42B30\\
Keywords: Intrinsic square functions; weighted weak Hardy spaces; $A_p$ weights; atomic decomposition
\end{abstract}

\section{Introduction}

Let ${\mathbb R}^{n+1}_+=\mathbb R^n\times(0,\infty)$ and $\varphi_t(x)=t^{-n}\varphi(x/t)$. The classical square function (Lusin area integral) is a familiar object. If $u(x,t)=P_t*f(x)$ is the Poisson integral of $f$, where $P_t(x)=c_n\frac{t}{(t^2+|x|^2)^{{(n+1)}/2}}$ denotes the Poisson kernel in ${\mathbb R}^{n+1}_+$. Then we define the classical square function (Lusin area integral) $S(f)$ by
\begin{equation*}
S(f)(x)=\bigg(\iint_{\Gamma(x)}\big|\nabla u(y,t)\big|^2t^{1-n}\,dydt\bigg)^{1/2},
\end{equation*}
where $\Gamma(x)$ denotes the usual cone of aperture one:
\begin{equation*}
\Gamma(x)=\big\{(y,t)\in{\mathbb R}^{n+1}_+:|x-y|<t\big\}
\end{equation*}
and
\begin{equation*}
\big|\nabla u(y,t)\big|=\left|\frac{\partial u}{\partial t}\right|^2+\sum_{j=1}^n\left|\frac{\partial u}{\partial y_j}\right|^2.
\end{equation*}
We can similarly define a cone of aperture $\beta$ for any $\beta>0$:
\begin{equation*}
\Gamma_\beta(x)=\big\{(y,t)\in{\mathbb R}^{n+1}_+:|x-y|<\beta t\big\},
\end{equation*}
and corresponding square function
\begin{equation*}
S_\beta(f)(x)=\bigg(\iint_{\Gamma_\beta(x)}\big|\nabla u(y,t)\big|^2t^{1-n}\,dydt\bigg)^{1/2}.
\end{equation*}
The Littlewood-Paley $g$-function (could be viewed as a ``zero-aperture" version of $S(f)$) and the $g^*_\lambda$-function (could be viewed as an ``infinite aperture" version of $S(f)$) are defined respectively by
\begin{equation*}
g(f)(x)=\bigg(\int_0^\infty\big|\nabla u(x,t)\big|^2 t\,dt\bigg)^{1/2}
\end{equation*}
and
\begin{equation*}
g^*_\lambda(f)(x)=\left(\iint_{{\mathbb R}^{n+1}_+}\bigg(\frac t{t+|x-y|}\bigg)^{\lambda n}\big|\nabla u(y,t)\big|^2 t^{1-n}\,dydt\right)^{1/2}.
\end{equation*}

The modern (real-variable) variant of $S_\beta(f)$ can be defined in the following way. Let $\psi\in C^\infty(\mathbb R^n)$ be real, radial, have support contained in $\{x:|x|\le1\}$, and $\int_{\mathbb R^n}\psi(x)\,dx=0$. The continuous square function $S_{\psi,\beta}(f)$ is defined by
\begin{equation*}
S_{\psi,\beta}(f)(x)=\bigg(\iint_{\Gamma_\beta(x)}\big|f*\psi_t(y)\big|^2\frac{dydt}{t^{n+1}}\bigg)^{1/2}.
\end{equation*}

In 2007, Wilson \cite{wilson1} introduced a new square function called intrinsic square function which is universal in a sense (see also \cite{wilson2}). This function is independent of any particular kernel $\psi$, and it dominates pointwise all the above defined square functions. On the other hand, it is not essentially larger than any particular $S_{\psi,\beta}(f)$. For $0<\alpha\le1$, let ${\mathcal C}_\alpha$ be the family of functions $\varphi$ defined on $\mathbb R^n$ such that $\varphi$ has support containing in $\{x\in\mathbb R^n: |x|\le1\}$, $\int_{\mathbb R^n}\varphi(x)\,dx=0$, and, for all $x, x'\in \mathbb R^n$,
\begin{equation*}
|\varphi(x)-\varphi(x')|\le|x-x'|^\alpha.
\end{equation*}
For $(y,t)\in {\mathbb R}^{n+1}_+$ and $f\in L^1_{{loc}}(\mathbb R^n)$, we set
\begin{equation*}
A_\alpha(f)(y,t)=\sup_{\varphi\in{\mathcal C}_\alpha}\big|f*\varphi_t(y)\big|.
\end{equation*}
Then we define the intrinsic square function of $f$ (of order $\alpha$) by the formula
\begin{equation*}
\mathcal S_\alpha(f)(x)=\left(\iint_{\Gamma(x)}\Big(A_\alpha(f)(y,t)\Big)^2\frac{dydt}{t^{n+1}}\right)^{1/2}.
\end{equation*}
We can also define varying-aperture versions of $\mathcal S_\alpha(f)$ by the formula
\begin{equation*}
\mathcal S_{\alpha,\beta}(f)(x)=\left(\iint_{\Gamma_\beta(x)}\Big(A_\alpha(f)(y,t)\Big)^2\frac{dydt}{t^{n+1}}\right)^{1/2}.
\end{equation*}
The intrinsic Littlewood-Paley $g$-function and the intrinsic $g^*_\lambda$-function will be defined respectively by
\begin{equation*}
g_\alpha(f)(x)=\left(\int_0^\infty\Big(A_\alpha(f)(x,t)\Big)^2\frac{dt}{t}\right)^{1/2}
\end{equation*}
and
\begin{equation*}
g^*_{\lambda,\alpha}(f)(x)=\left(\iint_{{\mathbb R}^{n+1}_+}\left(\frac t{t+|x-y|}\right)^{\lambda n}\Big(A_\alpha(f)(y,t)\Big)^2\frac{dydt}{t^{n+1}}\right)^{1/2}.
\end{equation*}

In \cite{wilson2}, Wilson proved the following result.

\newtheorem*{thma}{Theorem A}
\begin{thma}
Let $w\in A_p(\mbox{Muckenhoupt weight class})$, $1<p<\infty$ and $0<\alpha\le1$. Then there exists a constant $C>0$ independent of $f$ such that
$$\|\mathcal S_\alpha(f)\|_{L^p_w}\le C \|f\|_{L^p_w}.$$
\end{thma}

Moreover, in \cite{lerner}, Lerner showed sharp $L^p_w$ norm inequalities for the intrinsic square functions in terms of the $A_p$ characteristic constant of $w$ for all $1<p<\infty$. In \cite{huang}, Huang and Liu studied the boundedness of intrinsic square functions on the weighted Hardy spaces $H^1_w(\mathbb R^n)$. Furthermore, they obtained the intrinsic square function characterizations of $H^1_w(\mathbb R^n)$. Recently, in \cite{wang1} and \cite{wang2}, we have established the strong and weak type estimates of intrinsic square functions on the weighted Hardy spaces $H^p_w(\mathbb R^n)$ for $n/{(n+\alpha)}\le p<1$.

The main purpose of this paper is to investigate the mapping properties of intrinsic square functions on the weighted weak Hardy spaces $WH^p_w(\mathbb R^n)$ (see Section 2 for the definition). We now present our main results as follows.

\newtheorem{theorem}{Theorem}[section]
\begin{theorem}
Let $0<\alpha\le1$, $n/{(n+\alpha)}<p\le1$ and $w\in A_{p(1+\frac{\alpha}{n})}$. Then there exists a
constant $C>0$ independent of $f$ such that
\begin{equation*}
\big\|\mathcal S_\alpha(f)\big\|_{WL^p_w}\le C\|f\|_{WH^p_w}.
\end{equation*}
\end{theorem}

\begin{theorem}
Let $0<\alpha\le1$, $n/{(n+\alpha)}<p\le1$ and $w\in A_{p(1+\frac{\alpha}{n})}$. Suppose that $\lambda>{(3n+2\alpha)}/n$, then there exists a
constant $C>0$ independent of $f$ such that
\begin{equation*}
\big\|g^*_{\lambda,\alpha}(f)\big\|_{WL^p_w}\le C\|f\|_{WH^p_w}.
\end{equation*}
\end{theorem}

In \cite{wilson1}, Wilson also showed that for any $0<\alpha\le1$, the functions $g_\alpha(f)(x)$ and $\mathcal S_\alpha(f)(x)$ are pointwise comparable. Thus, as a direct consequence of Theorem 1.1, we obtain the following

\newtheorem{corollary}[theorem]{Corollary}
\begin{corollary}
Let $0<\alpha\le1$, $n/{(n+\alpha)}<p\le1$ and $w\in A_{p(1+\frac{\alpha}{n})}$. Then there exists a
constant $C>0$ independent of $f$ such that
\begin{equation*}
\big\|g_\alpha(f)\big\|_{WL^p_w}\le C\|f\|_{WH^p_w}.
\end{equation*}
\end{corollary}

\section{Notations and preliminaries}

\subsection{$A_p$ weights}

The definition of $A_p$ class was first used by Muckenhoupt \cite{muckenhoupt}, Hunt, Muckenhoupt and Wheeden \cite{hunt}, and Coifman and Fefferman \cite{coifman} in the study of weighted
$L^p$ boundedness of Hardy-Littlewood maximal functions and singular integrals. Let $w$ be a nonnegative, locally integrable function defined on $\mathbb R^n$; all cubes are assumed to have their sides parallel to the coordinate axes.
We say that $w\in A_p$, $1<p<\infty$, if
\begin{equation*}
\left(\frac1{|Q|}\int_Q w(x)\,dx\right)\left(\frac1{|Q|}\int_Q w(x)^{-1/{(p-1)}}\,dx\right)^{p-1}\le C \quad\mbox{for every cube}\; Q\subseteq \mathbb
R^n,
\end{equation*}
where $C$ is a positive constant which is independent of the choice of $Q$.

For the case $p=1$, $w\in A_1$, if
\begin{equation*}
\frac1{|Q|}\int_Q w(x)\,dx\le C\cdot\underset{x\in Q}{\mbox{ess\,inf}}\,w(x)\quad\mbox{for every cube}\;Q\subseteq\mathbb R^n.
\end{equation*}
A weight function $w\in A_\infty$ if it satisfies the $A_p$ condition for some $1<p<\infty$. It is well known that if $w\in A_p$ with $1<p<\infty$, then $w\in A_r$ for all $r>p$, and $w\in A_q$ for some $1<q<p$. We thus write $q_w\equiv\inf\{q>1:w\in A_q\}$ to denote the critical index of $w$.

Given a cube $Q$ and $\lambda>0$, $\lambda Q$ denotes the cube with the same center as $Q$ whose side length is $\lambda$ times that of $Q$. $Q=Q(x_0,r)$ denotes the cube centered at $x_0$ with side length $r$. For a weight function $w$ and a measurable set $E$, we denote the Lebesgue measure of $E$ by $|E|$ and set the weighted measure $w(E)=\int_E w(x)\,dx$.

We give the following results that will be used in the sequel.

\newtheorem{lemma}[theorem]{Lemma}
\begin{lemma}[\cite{garcia2}]
Let $w\in A_q$ with $q\ge1$. Then, for any cube $Q$, there exists an absolute constant $C>0$ such that
$$w(2Q)\le C\,w(Q).$$
In general, for any $\lambda>1$, we have
$$w(\lambda Q)\le C\cdot\lambda^{nq}w(Q),$$
where $C$ does not depend on $Q$ nor on $\lambda$.
\end{lemma}

\begin{lemma}[\cite{garcia2}]
Let $w\in A_q$ with $q>1$. Then, for all $r>0$, there exists a constant $C>0$ independent of $r$ such that
\begin{equation*}
\int_{|x|\ge r}\frac{w(x)}{|x|^{nq}}\,dx\le C\cdot r^{-nq}w\big(Q(0,2r)\big).
\end{equation*}
\end{lemma}

Given a weight function $w$ on $\mathbb R^n$, for $0<p<\infty$, we denote by $L^p_w(\mathbb R^n)$ the weighted space of all functions $f$ satisfying
\begin{equation*}
\|f\|_{L^p_w}=\bigg(\int_{\mathbb R^n}|f(x)|^pw(x)\,dx\bigg)^{1/p}<\infty.
\end{equation*}
When $p=\infty$, $L^\infty_w(\mathbb R^n)$ will be taken to mean $L^\infty(\mathbb R^n)$, and
\begin{equation*}
\|f\|_{L^\infty_w}=\|f\|_{L^\infty}=\underset{x\in\mathbb R^n}{\mbox{ess\,sup}}\,|f(x)|.
\end{equation*}
We also denote by $WL^p_w(\mathbb R^n)$ the weighted weak $L^p$ space which is formed by all measurable functions $f$ satisfying
\begin{equation*}
\|f\|_{WL^p_w}=\sup_{\lambda>0}\lambda\cdot w\big(\big\{x\in\mathbb R^n:|f(x)|>\lambda\big\}\big)^{1/p}<\infty.
\end{equation*}

\subsection{Weighted weak Hardy spaces}

Let us now turn to the weighted weak Hardy spaces. The (unweighted) weak $H^p$ spaces have first appeared in the work of Fefferman, Rivi\`ere and Sagher \cite{cfefferman}. The atomic decomposition theory of weak $H^1$ space on $\mathbb R^n$ was given by Fefferman and Soria in \cite{rfefferman}. Later, Liu \cite{liu} established the weak $H^p$ spaces on homogeneous
groups. For the boundedness properties of some operators on weak Hardy spaces, we refer the reader to [2--6] and [16]. In 2000, Quek and Yang \cite{quek} introduced the weighted weak Hardy spaces $WH^p_w(\mathbb R^n)$ and established their atomic decompositions. Moreover, by using the atomic decomposition theory of $WH^p_w(\mathbb R^n)$, Quek and Yang \cite{quek} also obtained the boundedness of Calder\'on-Zygmund type operators on these weighted spaces.

We write $\mathscr S(\mathbb R^n)$ to denote the Schwartz space of all rapidly decreasing smooth functions and $\mathscr S'(\mathbb R^n)$ to denote the space of all tempered distributions, i.e., the topological dual of $\mathscr S(\mathbb R^n)$. Let $w\in A_\infty$, $0<p\le1$ and $N=[n(q_w/p-1)]$. Define
\begin{equation*}
\mathscr A_{N,w}=\Big\{\varphi\in\mathscr S(\mathbb R^n):\sup_{x\in\mathbb R^n}\sup_{|\alpha|\le N+1}(1+|x|)^{N+n+1}\big|D^\alpha\varphi(x)\big|\le1\Big\},
\end{equation*}
where $\alpha=(\alpha_1,\dots,\alpha_n)\in(\mathbb N\cup\{0\})^n$, $|\alpha|=\alpha_1+\dots+\alpha_n$, and
\begin{equation*}
D^\alpha\varphi=\frac{\partial^{|\alpha|}\varphi}{\partial x^{\alpha_1}_1\cdots\partial x^{\alpha_n}_n}.
\end{equation*}
For $f\in\mathscr S'(\mathbb R^n)$, the grand maximal function of $f$ is defined by
\begin{equation*}
G_w f(x)=\sup_{\varphi\in\mathscr A_{N,w}}\sup_{|y-x|<t}\big|(\varphi_t*f)(y)\big|.
\end{equation*}
Then we can define the weighted weak Hardy space $WH^p_w(\mathbb R^n)$ by $WH^p_w(\mathbb R^n)=\big\{f\in\mathscr S'(\mathbb R^n):G_w f\in WL^p_w(\mathbb R^n)\big\}$. Moreover, we set $\|f\|_{WH^p_w}=\|G_w f\|_{WL^p_w}$.

\begin{theorem}[\cite{quek}]
Let $0<p\le1$ and $w\in A_\infty$. For every $f\in WH^p_w(\mathbb R^n)$, there exists a sequence of bounded measurable functions $\{f_k\}_{k=-\infty}^\infty$ such that

$(i)$ $f=\sum_{k=-\infty}^\infty f_k$ in the sense of distributions.

$(ii)$ Each $f_k$ can be further decomposed into $f_k=\sum_i b^k_i$, where $\{b^k_i\}$ satisfies

\quad $(a)$ Each $b^k_i$ is supported in a cube $Q^k_i$ with $\sum_{i}w(Q^k_i)\le c2^{-kp}$, and $\sum_i\chi_{Q^k_i}(x)\le c$. Here $\chi_E$ denotes the characteristic function of the set $E$ and $c\sim\big\|f\big\|_{WH^p_w}^p;$

\quad $(b)$ $\big\|b^k_i\big\|_{L^\infty}\le C2^k$, where $C>0$ is independent of $i$ and $k\,;$

\quad $(c)$ $\int_{\mathbb R^n}b^k_i(x)x^\alpha\,dx=0$ for every multi-index $\alpha$ with $|\alpha|\le[n({q_w}/p-1)]$.

Conversely, if $f\in\mathscr S'(\mathbb R^n)$ has a decomposition satisfying $(i)$ and $(ii)$, then $f\in WH^p_w(\mathbb R^n)$. Moreover, we have $\big\|f\big\|_{WH^p_w}^p\sim c.$
\end{theorem}

Throughout this article $C$ always denote a positive constant independent of the main parameters involved, but it may be different from line to line.

\section{Proof of Theorem 1.1}

Before proving our main theorem in this section, let us first establish the following lemma.

\begin{lemma}
Let $0<\alpha\le1$. Then for any given function $b\in L^\infty(\mathbb R^n)$ with support contained in $Q=Q(x_0,r)$, and $\int_{\mathbb R^n}b(x)\,dx=0$,
we have
\begin{equation*}
\mathcal S_\alpha(b)(x)\le C\cdot\|b\|_{L^\infty}\frac{r^{n+\alpha}}{|x-x_0|^{n+\alpha}}, \quad \mbox{whenever}\; \;|x-x_0|>\sqrt{n}r.
\end{equation*}
\end{lemma}

\begin{proof}
For any $\varphi\in{\mathcal C}_\alpha$, $0<\alpha\le1$, by the vanishing moment condition of $b$, we have that for any $(y,t)\in\Gamma(x)$,
\begin{align}
\big|(b*\varphi_t)(y)\big|&=\left|\int_Q\big(\varphi_t(y-z)-\varphi_t(y-x_0)\big)b(z)\,dz\right|\notag\\
&\le\int_Q\frac{|z-x_0|^\alpha}{t^{n+\alpha}}|b(z)|\,dz\notag\\
&\le C\cdot\|b\|_{L^\infty}\frac{r^{n+\alpha}}{t^{n+\alpha}}.
\end{align}
For any $z\in Q$, we have $|z-x_0|\le\frac{\sqrt n}{2}r<\frac{|x-x_0|}{2}$. Furthermore, we observe that $supp \,\varphi\subseteq\{x\in\mathbb R^n:|x|\le1\}$, then for any $(y,t)\in\Gamma(x)$, by a direct computation, we can easily see that
\begin{equation}
2t\ge|x-y|+|y-z|\ge|x-z|\ge|x-x_0|-|z-x_0|\ge\frac{|x-x_0|}{2}.
\end{equation}
Thus, for any point $x$ with $|x-x_0|>\sqrt{n}r$, it follows from the inequalities (1) and (2) that
\begin{equation*}
\begin{split}
\big|\mathcal S_\alpha(b)(x)\big|&=\left(\iint_{\Gamma(x)}\Big(\sup_{\varphi\in{\mathcal C}_\alpha}\big|(\varphi_t*b)(y)\big|\Big)^2\frac{dydt}{t^{n+1}}\right)^{1/2}\\
&\le C\cdot\|b\|_{L^\infty}r^{n+\alpha}\left(\int_{\frac{|x-x_0|}{4}}^\infty
\int_{|y-x|<t}\frac{dydt}{t^{2(n+\alpha)+n+1}}\right)^{1/2}\\
&\le C\cdot\|b\|_{L^\infty}r^{n+\alpha}\left(\int_{\frac{|x-x_0|}{4}}^\infty\frac{dt}{t^{2(n+\alpha)+1}}\right)^{1/2}\\
&\le C\cdot\|b\|_{L^\infty}\frac{r^{n+\alpha}}{|x-x_0|^{n+\alpha}}.
\end{split}
\end{equation*}
We are done.
\end{proof}

We are now in a position to give the proof of Theorem 1.1.

\begin{proof}[Proof of Theorem $1.1$]
For any given $\lambda>0$, we may choose $k_0\in\mathbb Z$ such that $2^{k_0}\le\lambda<2^{k_0+1}$. For every $f\in WH^p_w(\mathbb R^n)$, then by Theorem 2.3, we can write
\begin{equation*}
f=\sum_{k=-\infty}^\infty f_k=\sum_{k=-\infty}^{k_0} f_k+\sum_{k=k_0+1}^\infty f_k=F_1+F_2,
\end{equation*}
where $F_1=\sum_{k=-\infty}^{k_0} f_k=\sum_{k=-\infty}^{k_0}\sum_i b^k_i$, $F_2=\sum_{k=k_0+1}^\infty f_k=\sum_{k=k_0+1}^\infty\sum_i b^k_i$ and $\{b^k_i\}$ satisfies $(a)$--$(c)$ in Theorem 2.3. Then we have
\begin{equation*}
\begin{split}
&\lambda^p\cdot w\big(\big\{x\in\mathbb R^n:|\mathcal S_\alpha(f)(x)|>\lambda\big\}\big)\\
\le\,&\lambda^p\cdot w\big(\big\{x\in\mathbb R^n:|\mathcal S_\alpha(F_1)(x)|>\lambda/2\big\}\big)+\lambda^p\cdot w\big(\big\{x\in\mathbb R^n:|\mathcal S_\alpha(F_2)(x)|>\lambda/2\big\}\big)\\
=\,&I_1+I_2.
\end{split}
\end{equation*}
First we claim that the following inequality holds:
\begin{equation}
\big\|F_1\big\|_{L^2_w}\le C\cdot\lambda^{1-p/2}\big\|f\big\|^{p/2}_{WH^p_w}.
\end{equation}
In fact, since $supp\,b^k_i\subseteq Q^k_i=Q(x^k_i,r^k_i)$ and $\|b^k_i\|_{L^\infty}\le C 2^k$ by Theorem 2.3, then it follows from Minkowski's integral inequality that
\begin{equation*}
\begin{split}
\big\|F_1\big\|_{L^2_w}&\le\sum_{k=-\infty}^{k_0}\sum_i\big\|b^k_i\big\|_{L^2_w}\\
&\le\sum_{k=-\infty}^{k_0}\sum_i\big\|b^k_i\big\|_{L^\infty}w\big(Q^k_i\big)^{1/2}.
\end{split}
\end{equation*}
For each $k\in\mathbb Z$, by using the bounded overlapping property of the cubes $\{Q^k_i\}$ and the fact that $1-p/2>0$, we thus obtain
\begin{equation*}
\begin{split}
\big\|F_1\big\|_{L^2_w}&\le C\sum_{k=-\infty}^{k_0}2^k\Big(\sum_i w(Q^k_i)\Big)^{1/2}\\
&\le C\sum_{k=-\infty}^{k_0}2^{k(1-p/2)}\big\|f\big\|^{p/2}_{WH^p_w}\\
&\le C\sum_{k=-\infty}^{k_0}2^{(k-k_0)(1-p/2)}\cdot\lambda^{1-p/2}\big\|f\big\|^{p/2}_{WH^p_w}\\
&\le C\cdot\lambda^{1-p/2}\big\|f\big\|^{p/2}_{WH^p_w}.
\end{split}
\end{equation*}
Since $w\in A_{p(1+\frac{\alpha}{n})}$ and $1<p(1+\frac{\alpha}{n})\le1+\frac{\alpha}{n}\le2$, then $w\in A_2$. Hence, it follows from Chebyshev's inequality and Theorem A that
\begin{equation*}
\begin{split}
I_1&\le \lambda^p\cdot\frac{4}{\lambda^2}\big\|\mathcal S_\alpha(F_1)\big\|^2_{L^2_w}\\
&\le C\cdot\lambda^{p-2}\big\|F_1\big\|^2_{L^2_w}\\
&\le C\big\|f\big\|^{p}_{WH^p_w}.
\end{split}
\end{equation*}
Now we turn our attention to the estimate of $I_2$. If we set
\begin{equation*}
A_{k_0}=\bigcup_{k=k_0+1}^\infty\bigcup_i \widetilde{Q^k_i},
\end{equation*}
where $\widetilde{Q^k_i}=Q(x^k_i,\tau^{{(k-k_0)}/{(n+\alpha)}}(2\sqrt n)r^k_i)$ and $\tau$ is a fixed positive number such that $1<\tau<2$. Thus, we can further decompose $I_2$ as
\begin{equation*}
\begin{split}
I_2&\le\lambda^p\cdot w\big(\big\{x\in A_{k_0}:|\mathcal S_\alpha(F_2)(x)|>\lambda/2\big\}\big)+
\lambda^p\cdot w\big(\big\{x\in (A_{k_0})^c:|\mathcal S_\alpha(F_2)(x)|>\lambda/2\big\}\big)\\
&=I'_2+I''_2.
\end{split}
\end{equation*}
Since $w\in A_{p(1+\frac{\alpha}{n})}$, then by Lemma 2.1, we can get
\begin{equation*}
\begin{split}
I'_2&\le\lambda^p\sum_{k=k_0+1}^\infty\sum_iw\big(\widetilde{Q^k_i}\big)\\
&\le C\cdot\lambda^p\sum_{k=k_0+1}^\infty\tau^{(k-k_0)p}\sum_iw(Q^k_i)\\
&\le C\big\|f\big\|^{p}_{WH^p_w}\sum_{k=k_0+1}^\infty\Big(\frac{\tau}{2}\Big)^{(k-k_0)p}\\
&\le C\big\|f\big\|^{p}_{WH^p_w}.
\end{split}
\end{equation*}
On the other hand, an application of Chebyshev's inequality gives us that
\begin{equation*}
\begin{split}
I''_2&\le 2^p\int_{(A_{k_0})^c}\big|\mathcal S_\alpha(F_2)(x)\big|^pw(x)\,dx\\
&\le 2^p\sum_{k=k_0+1}^\infty\sum_i\int_{\big(\widetilde{Q^k_i}\big)^c}\big|\mathcal S_\alpha(b^k_i)(x)\big|^pw(x)\,dx.
\end{split}
\end{equation*}
When $x\in\big(\widetilde{Q^k_i}\big)^c$, then a direct calculation shows that $|x-x^k_i|\ge\tau^{{(k-k_0)}/{(n+\alpha)}}\sqrt n r^k_i>\sqrt n r^k_i$. Let $q=p(1+\frac{\alpha}{n})$ for simplicity. Then for any $n/{(n+\alpha)}<p\le1$, $w\in A_q$ and $q>1$, we can see that $[n(q_w/p-1)]=0$. Obviously, by Theorem 2.3, we know that all the functions $b^k_i$ satisfy the conditions in Lemma 3.1. Applying Lemma 2.2 and Lemma 3.1, we can deduce
\begin{equation*}
\begin{split}
I''_2&\le C\sum_{k=k_0+1}^\infty\sum_i2^{kp}(r^k_i)^{(n+\alpha)p}
\int_{|x-x^k_i|\ge\tau^{{(k-k_0)}/{(n+\alpha)}}\sqrt n r^k_i}\frac{w(x)}{|x-x^k_i|^{(n+\alpha)p}}\,dx\\
&= C\sum_{k=k_0+1}^\infty\sum_i2^{kp}(r^k_i)^{nq}
\int_{|y|\ge\tau^{{(k-k_0)}/{(n+\alpha)}}\sqrt n r^k_i}\frac{w_1(y)}{|y|^{nq}}\,dy\\
&\le C\sum_{k=k_0+1}^\infty\sum_i
2^{kp}\big(\tau^{{(k-k_0)}/{(n+\alpha)}}\big)^{-nq}w_1\big(Q(0,\tau^{{(k-k_0)}/{(n+\alpha)}}\cdot r^k_i)\big)\\
&= C\sum_{k=k_0+1}^\infty\sum_i
2^{kp}\big(\tau^{{(k-k_0)}/{(n+\alpha)}}\big)^{-nq}w\big(Q(x^k_i,\tau^{{(k-k_0)}/{(n+\alpha)}}\cdot r^k_i)\big),
\end{split}
\end{equation*}
where $w_1(x)=w(x+x^k_i)$ is the translation of $w(x)$. It is obvious that $w_1\in A_{q}$ whenever $w\in A_{q}$, and $q_{w_1}=q_w$. In addition, for $w\in A_q$ with $q>1$, then we can take a sufficiently small number $\varepsilon>0$ such that $w\in A_{q-\varepsilon}$. Therefore, by using Lemma 2.1 again, we obtain
\begin{equation*}
\begin{split}
I''_2&\le C\sum_{k=k_0+1}^\infty\sum_i2^{kp}\big(\tau^{{(k-k_0)}/{(n+\alpha)}}\big)^{-n\varepsilon}w(Q^k_i)\\
&\le C\big\|f\big\|^{p}_{WH^p_w}\sum_{k=k_0+1}^\infty\big(\tau^{{(k-k_0)}/{(n+\alpha)}}\big)^{-n\varepsilon}\\
&\le C\big\|f\big\|^{p}_{WH^p_w}.
\end{split}
\end{equation*}
Summing up the above estimates for $I_1$ and $I_2$ and then taking the supremum over all $\lambda>0$, we complete the proof of Theorem 1.1.
\end{proof}

\section{Proof of Theorem 1.2}

In order to prove Theorem 1.2, we shall need the following two lemmas.

\begin{lemma}
Let $0<\alpha\le1$, $n/{(n+\alpha)}<p\le1$ and $w\in A_{p(1+\frac{\alpha}{n})}$. Then for every $\lambda>p(1+\frac{\alpha}{n})$, we have
\begin{equation*}
\big\|g^*_{\lambda,\alpha}(f)\big\|_{L^2_w}\le C\|f\|_{L^2_w}
\end{equation*}
holds for all $f\in L^2_w(\mathbb R^n)$.
\end{lemma}

\begin{proof}
From the definition, we readily see that
\begin{align}
\big(g^*_{\lambda,\alpha}(f)(x)\big)^2=&\iint_{\mathbb R^{n+1}_+}\left(\frac{t}{t+|x-y|}\right)^{\lambda n}\Big(A_\alpha(f)(y,t)\Big)^2\frac{dydt}{t^{n+1}}\notag\\
=&\int_0^\infty\int_{|x-y|<t}\left(\frac{t}{t+|x-y|}\right)^{\lambda n}\Big(A_\alpha(f)(y,t)\Big)^2\frac{dydt}{t^{n+1}}\notag\\
&+\sum_{j=1}^\infty\int_0^\infty\int_{2^{j-1}t\le|x-y|<2^jt}\left(\frac{t}{t+|x-y|}\right)^{\lambda n}\Big(A_\alpha(f)(y,t)\Big)^2\frac{dydt}{t^{n+1}}\notag\\
\le&\, C\bigg[\mathcal S_\alpha(f)(x)^2+\sum_{j=1}^\infty 2^{-j\lambda n}\mathcal S_{\alpha,2^j}(f)(x)^2\bigg].
\end{align}
We are now going to estimate $\int_{\mathbb R^n}|\mathcal S_{\alpha,2^j}(f)(x)|^2w(x)\,dx$ for $j=1,2,\ldots.$
Fubini's theorem and Lemma 2.1 imply
\begin{align}
\int_{\mathbb R^n}\big|\mathcal S_{\alpha,2^j}(f)(x)\big|^2w(x)\,dx&=\iint_{{\mathbb R}^{n+1}_+}\bigg(\int_{|x-y|<2^j t}w(x)\,dx\bigg)\Big(A_\alpha(f)(y,t)\Big)^2\frac{dydt}{t^{n+1}}\notag\\
&\le C\cdot2^{j(n+\alpha)p}\iint_{{\mathbb R}^{n+1}_+}\bigg(\int_{|x-y|<t}w(x)\,dx\bigg)\Big(A_\alpha(f)(y,t)\Big)^2\frac{dydt}{t^{n+1}}\notag\\
&=C\cdot 2^{j(n+\alpha)p}\big\|\mathcal S_\alpha(f)\big\|^2_{L^2_w}.
\end{align}
Since $w\in A_{p(1+\frac{\alpha}{n})}$ and $1<p(1+\frac{\alpha}{n})\le2$, then we have $w\in A_2$. Therefore, under the assumption that $\lambda>p(1+\frac{\alpha}{n})$, it follows from Theorem A and the above inequalities (4) and (5) that
\begin{equation*}
\begin{split}
\big\|g^*_{\lambda,\alpha}(f)\big\|^2_{L^2_w}\le& \,C\bigg(\int_{\mathbb R^n}\big|\mathcal S_\alpha(f)(x)\big|^2w(x)\,dx+\sum_{j=1}^\infty 2^{-j\lambda n}\int_{\mathbb R^n}\big|\mathcal S_{\alpha,2^j}(f)(x)\big|^2w(x)\,dx\bigg)\\
\le& \,C\bigg(\big\|\mathcal S_\alpha(f)\big\|^2_{L^2_w}+\sum_{j=1}^\infty 2^{-j\lambda n}\cdot2^{j(n+\alpha)p}\big\|\mathcal S_\alpha(f)\big\|_{L^2_w}^2\bigg)\\
\le& \,C\cdot\big\|f\big\|^2_{L^2_w}\Big(1+\sum_{j=1}^\infty 2^{-j\lambda n}\cdot2^{j(n+\alpha)p}\Big)\\
\le& \,C\cdot\big\|f\big\|^2_{L^2_w}.
\end{split}
\end{equation*}
We are done.
\end{proof}

Following the same procedure as that of Lemma 3.1, we can also show

\begin{lemma}
Let $0<\alpha\le1$ and $j\in \mathbb Z_+$. Then for any given function $b\in L^\infty(\mathbb R^n)$ with support contained in $Q=Q(x_0,r)$, and $\int_{\mathbb R^n}b(x)\,dx=0$,
we have
\begin{equation*}
\mathcal S_{\alpha,2^j}(b)(x)\le C\cdot2^{j(3n+2\alpha)/2}\|b\|_{L^\infty}\frac{r^{n+\alpha}}{|x-x_0|^{n+\alpha}}, \quad \mbox{whenever}\; \;|x-x_0|>\sqrt{n}r.
\end{equation*}
\end{lemma}

\begin{proof}
For any $z\in Q(x_0,r)$, we have $|z-x_0|<\frac{|x-x_0|}{2}$. Then for all $(y,t)\in\Gamma_{2^j}(x)$ and $|z-y|\le t$ with $z\in Q$, as in the proof of Lemma 3.1, we can deduce that
\begin{equation}
t+2^jt\ge|x-y|+|y-z|\ge|x-z|\ge|x-x_0|-|z-x_0|\ge\frac{|x-x_0|}{2}.
\end{equation}
Thus, by using the inequalities (1) and (6), we obtain
\begin{equation*}
\begin{split}
\big|\mathcal S_{\alpha,2^j}(b)(x)\big|&=\left(\iint_{\Gamma_{2^j}(x)}\Big(\sup_{\varphi\in{\mathcal C}_\alpha}\big|(\varphi_t*b)(y)\big|\Big)^2\frac{dydt}{t^{n+1}}\right)^{1/2}\\
&\le C\cdot\|b\|_{L^\infty}r^{n+\alpha}\left(\int_{\frac{|x-x_0|}{2^{j+2}}}^\infty
\int_{|y-x|<2^jt}\frac{dydt}{t^{2(n+\alpha)+n+1}}\right)^{1/2}\\
&\le C\cdot2^{{jn}/2}\|b\|_{L^\infty}r^{n+\alpha}\left(\int_{\frac{|x-x_0|}{2^{j+2}}}^\infty\frac{dt}{t^{2(n+\alpha)+1}}\right)^{1/2}\\
&\le C\cdot2^{j(3n+2\alpha)/2}\|b\|_{L^\infty}\frac{r^{n+\alpha}}{|x-x_0|^{n+\alpha}}.
\end{split}
\end{equation*}
This finishes the proof of the lemma.
\end{proof}

We are ready to show our main theorem of this section.

\begin{proof}[Proof of Theorem $1.2$]
We follow the strategy of the proof of Theorem 1.1. For any given $\sigma>0$, we are able to choose $k_0\in\mathbb Z$ such that $2^{k_0}\le\sigma<2^{k_0+1}$. For every $f\in WH^p_w(\mathbb R^n)$, we can also write
\begin{equation*}
\begin{split}
&\sigma^p\cdot w\big(\big\{x\in\mathbb R^n:|g^*_{\lambda,\alpha}(f)(x)|>\sigma\big\}\big)\\
\le\,&\sigma^p\cdot w\big(\big\{x\in\mathbb R^n:|g^*_{\lambda,\alpha}(F_1)(x)|>\sigma/2\big\}\big)
+\sigma^p\cdot w\big(\big\{x\in\mathbb R^n:|g^*_{\lambda,\alpha}(F_2)(x)|>\sigma/2\big\}\big)\\
=\,&J_1+J_2,
\end{split}
\end{equation*}
where the notations $F_1$ and $F_2$ are the same as in the proof of Theorem 1.1. By our assumption, we know that $\lambda>{(3n+2\alpha)}/n>p(1+\frac{\alpha}{n})$. Applying Chebyshev's inequality, Lemma 4.1 and the previous inequality (3), we obtain
\begin{equation*}
\begin{split}
J_1&\le \sigma^p\cdot\frac{4}{\sigma^2}\big\|g^*_{\lambda,\alpha}(F_1)\big\|^2_{L^2_w}\\
&\le C\cdot\sigma^{p-2}\big\|F_1\big\|^2_{L^2_w}\\
&\le C\|f\|^{p}_{WH^p_w}.
\end{split}
\end{equation*}
To estimate the other term $J_2$, as before, we also set
\begin{equation*}
A_{k_0}=\bigcup_{k=k_0+1}^\infty\bigcup_i \widetilde{Q^k_i},
\end{equation*}
where $\widetilde{Q^k_i}=Q(x^k_i,\tau^{{(k-k_0)}/{(n+\alpha)}}(2\sqrt n)r^k_i)$, $\tau$ is also a fixed real number such that $1<\tau<2$ and $supp\,b^k_i\subseteq Q^k_i=Q(x^k_i,r^k_i)$. Again, we shall further decompose $J_2$ as
\begin{equation*}
\begin{split}
J_2&\le\sigma^p\cdot w\big(\big\{x\in A_{k_0}:|g^*_{\lambda,\alpha}(F_2)(x)|>\sigma/2\big\}\big)+
\sigma^p\cdot w\big(\big\{x\in (A_{k_0})^c:|g^*_{\lambda,\alpha}(F_2)(x)|>\sigma/2\big\}\big)\\
&=J'_2+J''_2.
\end{split}
\end{equation*}
Using the same arguments as that of Theorem 1.1, we can see that
\begin{equation*}
\begin{split}
J'_2&\le\sigma^p\sum_{k=k_0+1}^\infty\sum_iw\big(\widetilde{Q^k_i}\big)\\
&\le C\cdot\sigma^p\sum_{k=k_0+1}^\infty\tau^{(k-k_0)p}\sum_iw(Q^k_i)\\
&\le C\big\|f\big\|^{p}_{WH^p_w}.
\end{split}
\end{equation*}
Noting that $0<p\le1$. Then by Chebyshev's inequality and the inequality (4), we have
\begin{equation*}
\begin{split}
J''_2\le&\, 2^p\int_{(A_{k_0})^c}\big|g^*_{\lambda,\alpha}(F_2)(x)\big|^pw(x)\,dx\\
\le&\, 2^p\sum_{k=k_0+1}^\infty\sum_i\int_{\big(\widetilde{Q^k_i}\big)^c}\big|S_\alpha(b^k_i)(x)\big|^pw(x)\,dx\\
&+2^p\sum_{j=1}^\infty 2^{-j\lambda np/2}\sum_{k=k_0+1}^\infty\sum_i\int_{\big(\widetilde{Q^k_i}\big)^c}\big|S_{\alpha,2^j}(b^k_i)(x)\big|^pw(x)\,dx\\
=&\,K_0+\sum_{j=1}^\infty 2^{-j\lambda np/2}K_j.
\end{split}
\end{equation*}
Note that $[n(q_w/p-1)]=0$ by the hypothesis. Clearly, in view of Theorem 2.3, we can easily see that all the functions $b^k_i$ satisfy the conditions in Lemma 3.1 or Lemma 4.2. In the proof of Theorem 1.1, we have already showed that
\begin{equation*}
K_0\le C\big\|f\big\|^{p}_{WH^p_w}.
\end{equation*}
Below we shall give the
estimates of $K_j$ for every $j=1,2,\ldots$. Observe that for any $x\in\big(\widetilde{Q^k_i}\big)^c$, we have $|x-x^k_i|\ge\tau^{{(k-k_0)}/{(n+\alpha)}}\sqrt n r^k_i>\sqrt n r^k_i$. Since $\|b^k_i\|_{L^\infty}\le C 2^k$, then, using Lemma 4.2 and following the same lines as in Theorem 1.1, we can also deduce
\begin{equation*}
\begin{split}
K_j&\le C\cdot2^{j{(3n+2\alpha)}p/2}\sum_{k=k_0+1}^\infty\sum_i\big\|b^k_i\big\|^p_{L^\infty}(r^k_i)^{(n+\alpha)p}
\int_{|x-x^k_i|\ge\tau^{{(k-k_0)}/{(n+\alpha)}}\sqrt n r^k_i}\frac{w(x)}{|x-x^k_i|^{(n+\alpha)p}}\,dx\\
&\le C\cdot2^{j{(3n+2\alpha)}p/2}\sum_{k=k_0+1}^\infty\sum_i2^{kp}(r^k_i)^{(n+\alpha)p}
\int_{|x-x^k_i|\ge\tau^{{(k-k_0)}/{(n+\alpha)}}\sqrt n r^k_i}\frac{w(x)}{|x-x^k_i|^{(n+\alpha)p}}\,dx\\
&\le C\cdot2^{j{(3n+2\alpha)}p/2}\big\|f\big\|^{p}_{WH^p_w}.
\end{split}
\end{equation*}
Hence, we finally obtain
\begin{equation*}
\begin{split}
J''_2&\le C\big\|f\big\|^{p}_{WH^p_w}\bigg(1+\sum_{j=1}^\infty2^{-j\lambda np/2}\cdot2^{j(3n+2\alpha)p/2}\bigg)\\
&\le C\big\|f\big\|^{p}_{WH^p_w},
\end{split}
\end{equation*}
where the last series is convergent since $\lambda>{(3n+2\alpha)}/n$. Therefore, combining the above estimates for $J_1$ and $J_2$ and then taking the supremum over all $\sigma>0$, we conclude the proof of Theorem 1.2.
\end{proof}

\end{document}